\documentclass[12pt,draftcls,onecolumn]{IEEEtran}

\usepackage{graphicx}
\usepackage{amsmath}
\usepackage{amssymb}
\usepackage{cite}
\usepackage{enumerate}

\ifCLASSINFOpdf

\else

\fi

\begin{document}

\newcommand{\qed}{\hfill \mbox{\raggedright \rule{.08in}{.086in}}}

\title{Linear Consensus Algorithms Based on Balanced Asymmetric Chains}

\author{Sadegh Bolouki and Roland P. Malham\'e
\thanks{GERAD and Department of Electrical Engineering, \'Ecole Polytechniqe de Montr\'eal, Montreal,
Canada, e-mail: \{sadegh.bolouki,roland.malhame\}@polymtl.ca}}

\maketitle


\begin{abstract}
Multi agent consensus algorithms  with update steps based on so-called balanced asymmetric chains, are analyzed. For such algorithms it is shown that (i) the set of accumulation points of states is finite, (ii) the asymptotic unconditional occurrence of single consensus or multiple consensuses is directly related to the property of absolute infinite flow for the underlying update chain. The results are applied to well known consensus models.
\end{abstract}


\IEEEpeerreviewmaketitle


\section{Introduction}

\newtheorem{theorem}{Theorem}
\newtheorem{definition}{Definition}
\newtheorem{lemma}{Lemma}
\newtheorem{corollary}{Corollary}
\newtheorem{remark}{Remark}
\newtheorem{proposition}{Proposition}

Consensus problems in multi-agent systems have gained increasing attention in various research communities.
Many of the consensus algorithms in the literature can be described by linear update equations:
\begin{equation}
  X(n+1) = A_n X(n), \,\,\, n \geq 0,
  \label{modelss}
\end{equation}
where $X(n)$ is the vector of states (the value of an unknown parameter or probability) and $A_n$ for every $n \geq 0$ is a stochastic matrix, i.e., each row of $A_n$ sums to 1. $A_n$ will be referred to as the matrix of interaction coefficients. Distributed averaging algorithms were first introduced by DeGroot in \cite{DeGroot:74}. Later, Chatterjee and Senata \cite{Chatterjee:77} considered the same class of consensus problems with time-varying interaction coefficients. The authors found sufficient conditions for consensus via backward products of stochastic matrices. Results of \cite{Chatterjee:77} were generalized in \cite{Tsit:84,Tsit:86,Tsit:89a}, whereby more general conditions for consensus to occur were provided. Unlike \cite{DeGroot:74,Chatterjee:77}, in the model considered in \cite{Tsit:84,Tsit:86,Tsit:89a}, communication links between individuals are not necessarily bidirectional. Briefly stated, sufficient conditions for convergence in \cite{Tsit:86,Tsit:84,Tsit:89a} are, non vanishing interaction rates, and continuously repeated connectivity of the integrated communication graph. 
As an alternative model, Vicsek et al. \cite{Vicsek:95} considered a system of multiple agents moving in the plane with the same speed but different headings, where heading of agents are updated according to an averaging algorithm. Consensus was observed in simulations. Jadbabaie et al. in \cite{Jadba:03} analyzed a linearized version of the Viscek model and provided conditions under which consensus occurs. The authors showed that consensus occurs exponentially fast if there exists an infinite sequence of contiguous, nonempty, bounded, time-intervals $[n_i,n_{i+1})$, $i \geq 0$, starting at $n_0=0$, with the property that across each such interval, all agents are linked together (via a chain of neighbors). Following \cite{Jadba:03}, many authors tried to generalize the consensus results by employing different techniques (see \cite{Hend:11} and references therein). 
Hendrickx et al. in recent work \cite{Hend:11} generalized the previous results by introducing an important property of stochastic matrices, the so-called cut-balance property. The authors also considered the multiple consensus problem. However, to obtain the main results, in the discrete case, a uniform positive lower bound for non zero interaction coefficients still appeared to be necessary, unlike in the corresponding continuous time theorems. Recently, Touri and Nedi\'c \cite{Touri:10a,Touri:10b,Touri:3,Touri:11c} have approached the consensus problem via the backward product of stochastic matrices as in \cite{Chatterjee:77}. For a class of random stochastic matrices, they have derived necessary and sufficient conditions for a.s. ergodicity. Existing results on consensus in discrete time distributed averaging algorithms are subsumed in \cite{Touri:10b,Touri:11c}.

In this note, by introducing a property of stochastic chains, herein called balanced asymmetry, we derive equivalent conditions for unconditional consensus and multiple consensus to occur in a class of multi-agent systems with dynamics (\ref{modelss}). In the process, we also establish that if the balanced asymmetry property is satisfied, the set of accumulation points of states is finite.

The rest of this paper is organized as follows: Essential notions that are required to state the main results are defined and illustrated in Section II. Main results on unconditional consensus and multiple consensus are presented in Section III. The relationship of our results to existing results in the literature as well as their applications to known models are discussed in Section IV. Concluding remarks end the paper in Section V.

\subsection{Notation}

\label{notations}

  Throughout this article, we adopt the following notation:

  \begin{itemize}

  \item $S$ is the set of agents and $s = |S|$ is the number of agents.

  \item $n$ stands for discrete time index.

  \item $X(n) = [X_1(n) \cdots X_s(n)]^T$, $n \geq 0$, is the state vector.

  \item For every $n \geq 0$, $(1_n,2_n,\ldots,s_n)$ is a permutation of $\{1,2,\ldots,s\}$ such that agent $i_n$ ($1 \leq i \leq s$) has the $i$th least state value among all agents at time $n$.

  \item $z_i(n)=X_{i_n}(n)$ is the $i$th least number among $X_1(n),\ldots,X_s(n)$. Particularly, $z_1(n)$ and $z_s(n)$ are the state values of agents associated with the least and the greatest state values at time $n$ respectively.

  \item $A_n$, $n \geq 0$ is the matrix of interaction rates $a_{ij}(n)$, $1 \leq i,j \leq s$.


  \end{itemize}

\section{Notions and Terminology}

\begin{definition}
  Consider a multi-agent system with dynamics (\ref{modelss}). By unconditional consensus in system (\ref{modelss}), we mean that no matter at what instant or at what values states are initialized, all $X_i(n)$'s, $i=1,\ldots,s$, converge to identical values as $n$ goes to infinity.
\end{definition}
We now define ergodicity according to \cite{Touri:10a}. Let $(A_n)$ be a chain of stochastic matrices. For $n > k \geq 0$, following \cite{Touri:10a}, denote $A(n,k) = A_{n-1}A_{n-2}\ldots A_k$.
\begin{definition} \cite{Touri:10a}
  A chain $(A_n)$ of stochastic matrices is said to be \textit{ergodic} if for every $k \geq 0$, $\lim_{n \rightarrow \infty} A(n,k)$ exists and is equal to a matrix with identical rows.
\end{definition}
It is possible to show that occurrence of unconditional consensus in a multi-agent system is equivalent to ergodicity of the transition chain of the system. This is how unconditional consensus and ergodicity are related. Besides consensus, there is another important notion, multiple consensus, that constitutes our focus in this work.
\begin{definition}
  For a multi-agent system with dynamics (\ref{modelss}), unconditional multiple consensus occurs if for every $i$, $1 \leq i \leq s$, $\lim_{n \rightarrow \infty} X_i(n)$ exists, no matter at what instant or at what values states are initialized.
\end{definition}
To formulate multiple consensus as a property of chains of stochastic matrices, we introduce class-ergodicity, as follows.
\begin{definition}
A chain $(A_n)$ of stochastic matrices is \textit{class-ergodic} if $\forall k \geq 0$, $\lim_{n \rightarrow \infty} A(n,k)$ exists and can be relabeled as a block diagonal matrix with each block having identical rows. By relabeling, we mean applying the same permutation to rows and columns of a square matrix.
\end{definition}
Clearly, if $(A_n)$ in dynamics (\ref{modelss}) is class-ergodic, unconditional multiple consensus occurs. The converse is true also, by noting that the $i$th column of $A(n,k)$ is equal to $X(n)$ when $X$ is initialized at time $k$ by the initial value $e_i$ denoting all of the components equal to zero, but the $i$th one equal to 1. Therefore, unconditional multiple consensus occurs in a system with dynamics (\ref{modelss}) if and only if chain $(A_n)$ is class-ergodic.

In the rest of this section, we provide essential notions that are employed to obtain our main results.

\subsection{$l_1$-approximation \cite{Touri:10b}}

\begin{definition}
  Chain $(A_n)$ is said to be an $l_1$-approximation of chain $(B_n)$ if $\sum_{n=0}^{\infty} \| A_n-B_n \|$ is finite, where the norm refers to the \textit{max norm}, i.e., the maximum of the absolute values of the matrix entries.
\end{definition}
It is not difficult to show that $l_1$-approximation is an equivalence relation in the set of chains of row stochastic matrices.
\begin{proposition}\cite{Touri:10b}
  Let $(A_n)$ be an $l_1$-approximation of chain $(B_n)$. Then, $(A_n)$ is class-ergodic if and only if $(B_n)$ is.
  \label{partial ergodicity}
\end{proposition}

\subsection{Absolute Infinite Flow \cite{Touri:3}}


\begin{definition}
  A chain $(A_n)$ of row stochastic matrices is said to have the absolute infinite flow property if
  \begin{equation}
       \sum_{n=0}^{\infty} \Big( \sum_{i \in T(n+1)} \sum_{j \in \bar{T}(n)} a_{ij}(n) + \sum_{i \in \bar{T}(n+1)} \sum_{j \in T(n)} a_{ij}(n)\Big) = \infty
    \label{abs}
  \end{equation}
  where $T(0),T(1),\ldots$ is an arbitrary sequence of subsets of $S$, $S = \{ 1,\ldots,s \}$, with the same cardinality, and $\bar{T}_i$ denotes the complement of $T_i$ in $S$.
Note that if $A_n$ is a matrix of order 1, i.e., $s=1$, then the absolute infinite flow property is trivially satisfied.
\end{definition}
In \cite{Touri:3}, the authors show that the absolute infinite flow property is a necessary condition for ergodicity. In addition, they prove necessity and sufficiency of the absolute infinite property in the case of chains of doubly stochastic matrices.

\subsection{Balanced Asymmetry}
\label{bal-asym}

\begin{definition}
  Consider a chain $(A_n)$ of stochastic matrices. Chain $(A_n)$ is said to be \textit{balanced asymmetric} if there exists an $M \geq 1$ such that for any two non empty subsets $S_1$ and $S_2$ of $S=\{1,\ldots,s\}$ with the same cardinality, we have
  \begin{equation}
    \sum_{i \in S_1} \sum_{j \in \bar{S}_2} a_{ij}(n) \leq M \sum_{i \in \bar{S}_1} \sum_{j \in S_2} a_{ij}(n), \,\,\,\,\,\forall n \geq 0,
    \label{b-asym}
  \end{equation}
  where the overbar indicates complementation.
\end{definition}
We provide the following non~trivial subclasses of balanced asymmetric chains:
\begin{enumerate}
  \item \textit{chains of doubly stochastic matrices:} It can be shown that all chains of doubly stochastic matrices are balanced asymmetric with $M=1$.

      \item Chains possessing the following two properties:\\
      \textit{self-confidence:} There exists $\delta > 0$ such that $a_{ii}(n) \geq \delta$ for every $i=1,\ldots,s$ and $n \geq 0$.\\
      \textit{cut-balance:} \cite{Hend:11} There exists $K \geq 1$, such that for every $E \subset \{1,\ldots,s\}$
  \begin{equation}
    \sum_{i \in E} \sum_{j \in \bar{E}} a_{ij}(n) \leq K \sum_{i \in \bar{E}} \sum_{j \in E} a_{ij}(n), \,\,\,\,\,\forall n \geq 0.
    \label{type-symmetry}
  \end{equation}
   Indeed, inequalities (\ref{type-symmetry}) and (\ref{b-asym}) are equivalent when $S_1$ is identical to $S_2$, while if $S_1 \neq S_2$, then $S_1 \cap \bar{S_2}$ and $\bar{S_1} \cap S_2$ are both non empty. As a result, and given the assumed self confidence property, both sums in inequality (\ref{b-asym}) are bounded below by $\delta$. In addition, both sums are bounded above by $s-1$ for any non empty $S_i$, $i =1,2$. Thus, the chain is balanced asymmetric with $M = \max\{K,(s-1)/\delta\}$. Note that the cut-balance property defined above, is the definition given in \cite{Hend:11} in the continuous time case. In \cite{Touri:11c}, chains having the cut-balance property are called balanced chains.
\end{enumerate}
\begin{remark}
  Balanced asymmetry is a stronger condition than cut-balance, although the latter together with self-confidence, becomes stronger than the former.
\end{remark}
\begin{remark}
  For those chains that are $l_1$-approximation of balanced asymmetric chains, the absolute infinite flow property is equivalent to:
  \begin{equation}
    \sum_{n=0}^{\infty} \sum_{i \in \bar{T}(n+1)} \sum_{j \in T(n)} a_{ij}(n) = \infty
    \label{abs-eq}
  \end{equation}
  for any sequence $T(n)$ of subsets of $T$ as in Eq. (\ref{abs}). This can be easily seen by combining relations (\ref{abs}) and (\ref{b-asym}).
\label{rmrk2}
\end{remark}

\subsection{Unbounded Interactions Graph \cite{Hend:11}}
\label{persistency}
The unbounded interactions graph of a chain is an important notion in this article, especially in class-ergodicity analysis. In the following, we define unbounded interactions graph of a chain of row stochastic matrices, which is the discrete time version of the definition given in \cite{Hend:11}.
\begin{definition}
  Let $(A_n)$ be a stochastic chain representing interaction coefficients of $s$ agents, where $S=\{ 1,\ldots,s\}$ is the set of agents. We form a directed graph $G_A=\{S,E\}$ with $(i,j) \in E$ if and only if $\sum_{n=0}^{\infty} a_{ij}(n) = \infty$. $G_A$ is called the \textit{unbounded interactions graph} of $A$.
\end{definition}
Taking into account that balanced asymmetry is a stronger condition than cut-balance, following a proof quite similar to that of Theorem 2 (b) in \cite{Hend:11}, one can establish the following proposition.
\begin{proposition}
  Let $(A_n)$ be stochastic chain with unbounded interactions graph $G_A$. If $(A_n)$ is balanced asymmetric, then every weakly connected component of $G_A$ is strongly connected.
\end{proposition}
According to Proposition 2, under the balanced asymmetry condition, the unbounded interactions graph can be partitioned into strongly connected components, herein called \textit{islands}.


\section{Convergence Results}

Recalling the definition of $z_i(n)$'s from Part \ref{notations}, we first state a theorem on the limiting behavior of states in a multi-agent system associated with an $l_1$-approximation of a balanced asymmetric chain.
\begin{theorem}
  Consider a multi-agent system with dynamics (\ref{modelss}). Assume that chain $(A_n)$ is an $l_1$-approximation of a balanced asymmetric chain. Then, $\lim_{n \rightarrow \infty}z_i(n)$ exists for every $i \in S$.
  \label{result-bal}
\end{theorem}
\begin{proof}
  To prove Theorem \ref{result-bal}, we use a technique similar to the one we adopted previously in proving Theorem 2 of \cite{Bolouki:ifac}. Note that this technique was also independently discovered by Hendrickx and Tsitsiklis (see \cite{Hend:11}). According to the definition of $z_i(n)$, we have $z_1(n) \leq z_2(n) \leq \cdots \leq z_s(n)$, $\forall n \geq 0$. Moreover, since states of agents are updated via a convex combination of their current states, $z_1(n)$ is a non-decreasing function of $n$, and $z_s(n)$ in a non-increasing function of $n$. Thus,
  \begin{equation}
    z_1(0) \leq z_i(n) \leq z_s(0), \,\,\, \forall i \in S.
  \end{equation}
  As a result, both $z_i(n)$ and $X_i(n)$ are bounded in a bounded interval, and defining $L \triangleq z_s(0)-z_1(0)$, we have:
  \begin{equation}
    X_i(n) - X_j(n) \leq L, \,\,\, \forall n \geq 0, \, \forall i,j \in S.
  \end{equation}
  Now, let $(B_n)$ be a balanced asymmetric chain that is an $l_1$-approximation of $(A_n)$. Let $A_n = B_n + M_n$, $\forall n \geq 0$. Denote $\|M_n\| \triangleq m_n$, $n \geq 0$, and $m'_n \triangleq \sum_{k=0}^{n-1} m_k$, $n > 0$ with $m'_0 = 0$. Note that $m'_n$ remains bounded, according to the definition of $l_1$-approximation. Set $K=2M$, and recalling $L \triangleq z_s(0)-z_1(0)$, define function $S_r(n)$ for every $r$, $1 \leq r \leq s$ by
  \begin{equation}
    S_r(n) \triangleq \sum_{i=1}^r K^{-i}(z_i(n) + sm'_nL).
    \label{S_r}
  \end{equation}
	In the following we show that $\lim_{n\rightarrow \infty}S_r(n)$ exists for every $r=1,\ldots,s$. Since $S_r$ is a linear combination of $z_i$'s with bounded coefficients, and $m'_n$ is bounded, it is bounded. Moreover,
  \begin{equation}
    S_r(n+1) - S_r(n) \geq K^{-s} \sum_{k=1}^{r-1} \left[ \left(\sum_{i=k+1}^{s} \sum_{j=1}^k b_{i_{n+1}j_n}\right) \left( z_{k+1}(n) - z_k(n) \right)\right] \geq 0
    \label{ee7}
  \end{equation}
	(see \cite{Bolouki:12b} for details). Hence, $S_r(n)$ is non decreasing. From boundedness and monotonic increasing behavior of $S_r$, we obtain that $\lim_{n\rightarrow \infty}S_r(n)$ exists for every $r=1,\ldots,s$. Furthermore, defining $S_0 \equiv 0$, Eq. (\ref{S_r}) implies
  \begin{equation}
    z_i(n) = K^i (S_i(n) - S_{i-1}(n))- sm'_iL.
  \end{equation}
  Thus, convergence of $z_i$'s is immediately implied from convergence of $S_i$, $S_{i-1}$, and $m'_i$.
\end{proof}
Convergence of $z_i(n)$'s in Theorem \ref{result-bal} implies that the set of accumulation points of agents' states is finite. In the next two theorems, we address the issues of unconditional consensus (ergodicity) and unconditional multiple consensus (class-ergodicity).


\begin{theorem}
  If chain $(A_n)$ is an $l_1$-approximation of a balanced asymmetric chain, then $(A_n)$ is ergodic if and only if it has the absolute infinite flow property.
  \label{ergodic}
\end{theorem}

\begin{proof}
  The necessity of the absolute infinite flow property has been proved in \cite{Touri:3}. Here we show that if chain $(A_n)$ has the absolute infinite flow property together with being an $l_1$-approximation of a balanced asymmetric chain, then $(A_n)$ is ergodic, or equivalently, consensus occurs in system (\ref{modelss}), no matter at what instant or what values states are initialized. With no loss of generality, we assume that states are initialized at $n=0$ (Otherwise, if states are initialized at $n=n_0 \neq 0$, we remove the first $n_0$ term of $(A_n)$ and obtain another chain which is still an $l_1$-approximation of a balanced asymmetric chain and has the absolute infinite flow property, and proceed with the new chain).
Let $(B_n)$ be a balanced asymmetric chain with bound $M$ which is an $l_1$-approximation of $(A_n)$. It is straightforward to verify that chain $(A_n)$ has the absolute infinite flow property if and only if chain $(B_n)$ does. The main part of the proof is common with the proof of Theorem 1. According to Theorem 1, we know that $\lim_{n \rightarrow \infty} z_i(n)$ exists for every $i \in S$. Let us define $\forall i \in S$: $Z_i = \lim_{n \rightarrow \infty} z_i(n)$. From the definition of $z_i$'s, we have:
  \begin{equation}
    Z_1 \leq Z_2 \leq \cdots \leq Z_s
    \label{z inequality 2}
  \end{equation}
  Since $z_1(n)$ and $z_s(n)$ are respectively the least and the greatest values of states at time $n$, consensus occurs if and only if $Z_1 = Z_s$. Assume that this does not happen, or equivalently, $Z_1 < Z_s$. We wish to show that applying the absolute infinite flow property in inequality (\ref{ee7}) when $r=s$, leads to an unbounded $S_s(n)$, which would be a contradiction. Since $Z_1 < Z_s$, from inequalities (\ref{z inequality 2}) we conclude that there exists $p$, $1 \leq p \leq s-1$ such that $Z_p < Z_{p+1}$. If we define $\epsilon \triangleq (Z_{p+1}-Z_p)/2 > 0$, there exists $N \geq 0$ such that
  \begin{equation}
    z_{p+1}(n) - z_p(n) > \epsilon, \,\,\, \forall n \geq N
    \label{asd1}
  \end{equation}
  On the other hand, for balanced asymmetric chains, the absolute infinite flow property reduces to Eq. (\ref{abs-eq}). From Eq. (\ref{abs-eq}), we conclude that for any sequence $T(n)$ of subsets of $S$ of the same cardinality:
  \begin{equation}
    \sum_{n=N}^{\infty} \sum_{i \in \bar{T}(n+1)} \sum_{j \in T(n)} b_{ij}(n) = \infty
    \label{asd2}
  \end{equation}
  since $\sum_{n=0}^{N-1} \sum_{i \in \bar{T}(n+1)} \sum_{j \in T(n)} b_{ij}(n)$ is finite. If in Eq. (\ref{asd2}) we set $T(n) = \{ 1_n,2_n,\ldots,r_n \}$, we obtain
  \begin{equation}
    \sum_{n=N}^{\infty} \sum_{i=r+1}^{s} \sum_{j=1}^r b_{i_{n+1}j_n} = \infty
    \label{asd3}
  \end{equation}
  On the other hand, we note that according to Theorem 1, $\lim_{n \rightarrow \infty}S_r(n)$ exists for every $r = 1,\ldots,s$. Therefore, we can write
  \begin{equation}
    \lim_{n \rightarrow \infty} S_r(n) - S_r(0) = \sum_{n=0}^{\infty} \left(S_r(n+1) - S_r(n)\right)
    \label{asd4}
  \end{equation}
  Relations (\ref{asd4}) and (\ref{ee7}) yield:
  \begin{equation}
    \begin{array}{ll}
      & \hspace{-.25in}\lim_{n \rightarrow \infty} S_r(n) - S_r(0) \geq \sum_{n=0}^{\infty}\left\{ K^{-s} \sum_{k=1}^{r-1}\left[ \left(\sum_{i=k+1}^{s} \sum_{j=1}^k b_{i_{n+1}j_n}\right)\left( z_{k+1}(n) - z_k(n) \right)\right]\right\}\vspace{.05in}\\
      &\hspace{1.32in} =  K^{-s} \sum_{k=1}^{r-1} \left[ \sum_{n=0}^{\infty} \left( \sum_{i=k+1}^{s} \sum_{j=1}^k b_{i_{n+1}j_n}\right)\left( z_{k+1}(n) - z_k(n) \right) \right]
    \end{array}
  \end{equation}
  Setting $r=s$ we obtain
  \begin{equation}
    \lim_{n \rightarrow \infty} S_s(n) - S_s(0)
      \geq  K^{-s} \sum_{k=1}^{s-1} \left[ \sum_{n=0}^{\infty} \left( \sum_{i=k+1}^{s} \sum_{j=1}^k b_{i_{n+1}j_n}\right)
      \left( z_{k+1}(n) - z_k(n) \right) \right]
    \label{asd5}
  \end{equation}
  From the above inequality, recalling that $z_{k+1}(n) \geq z_k(n)$, and keeping only terms corresponding to $k=p$ and $n \geq N$ in the RHS, we obtain
  \begin{equation}
    \lim_{n \rightarrow \infty} S_s(n) - S_s(0)
      \geq K^{-s} \sum_{n=N}^{\infty} \left( \sum_{i=p+1}^{s} \sum_{j=1}^p b_{i_{n+1}j_n}\right)
      \left( z_{p+1}(n) - z_p(n) \right)
    \label{asd6}
  \end{equation}
  Inequalities (\ref{asd1}) and (\ref{asd6}) imply
  \begin{equation}
      \lim_{n \rightarrow \infty} S_s(n) - S_s(0)
      \geq \epsilon . K^{-s}\sum_{n=N}^{\infty} \sum_{i=p+1}^{s} \sum_{j=1}^p b_{i_{n+1}j_n}
    \label{asd7}
  \end{equation}
  From Eq. (\ref{asd3}) we know that the RHS of inequality (\ref{asd7}) is unbounded. Thus, the LHS is unbounded, and so is $S_s(n)$, which is a contradiction. This completes the proof.
\end{proof}


\begin{theorem}
  Let chain $(A_n)$ be an $l_1$-approximation of a balanced asymmetric chain. Then, $(A_n)$ is class-ergodic if and only if the absolute infinite flow property holds over each island of the unbounded interactions graph induced by $(A_n)$.
  \label{result-semi}
\end{theorem}

\begin{proof}
  To prove the sufficiency of the condition, we adopt the same technique as used in \cite{Touri:10b} and form a new chain $(B_n)$ of the bounded interactions graph $G_A$ by eliminating interaction coefficients between each agent within an island and agents of other islands at all times. From definition of $G_A$ and its islands, it is immediately implied that $(B_n)$ is an $l_1$-approximation of $(A_n)$, and consequently, is an $l_1$-approximation of a balanced asymmetric chain. According to Proposition \ref{partial ergodicity}, it suffices to prove that $(B_n)$ is class-ergodic. The system with $(B_n)$ as transition chain can be decomposed into subsystems corresponding to islands, as there is no communication between islands at all. It is straightforward to verify that each subchain of $(B_n)$ corresponding to a subsystem is balanced asymmetric and possesses the absolute infinite flow property. Thus, Theorem \ref{ergodic} implies that each subchain is ergodic, and as a result, $(B_n)$ is class-ergodic.

  We now prove the converse property. More specifically, we assume that $(A_n)$ is class-ergodic and also is an $l_1$-approximation of a balanced asymmetric chain, and prove that the absolute infinite flow property holds inside each island. Once again we form chain $(B_n)$ from $(A_n)$ by eliminating all interaction coefficients between agents of distinct islands. Since $(B_n)$ is an $l_1$-approximation of $(A_n)$, Proposition \ref{partial ergodicity} implies that $(B_n)$ is class-ergodic as well. Note that $(B_n)$ is also an $l_1$-approximation of a balanced asymmetric chain, as $(A_n)$ is. It is sufficient now to show that the absolute infinite flow property holds inside islands of the bounded interactions graph induced by chain $(B_n)$. Define subchains of $(B_n)$ corresponding to islands. We shall show that each island subchain is ergodic. Thus, consider an arbitrary initial state for each subsystem and by concatenating these states, form an initial vector $Y(0)$ for the original system:
  \begin{equation}
    Y(n+1) = B_n Y(n),\,\,\, n \geq 0.
    \label{final sys}
  \end{equation}
  Since $(B_n)$ is assumed class-ergodic, unconditional multiple consensus occurs in system (\ref{final sys}). Let $I$ be an arbitrary island. We wish to show that agents of $I$ belong to the same consensus cluster. Assume that on the contrary, there exists an island $I$ containing agents corresponding to distinct consensus clusters. We proceed with the exact same proof of Theorem \ref{ergodic}, identifying this time $Y$ with $X$ in the theorem, and taking advantage of inequality (\ref{asd7}) by setting $p$ as follows: since members of island $I$ do not belong to the same cluster, $I$ can be partitioned into non empty $I_1$ subsets and $\bar{I}_1$ such that
  \begin{equation}
    \lim_{n \rightarrow \infty} Y_i(n) < \lim_{n \rightarrow \infty} Y_j(n), \,\,\, \forall i \in I_1, \, j \in \bar{I}_1.
  \end{equation}
  Recalling that $(B_n)$ is an $l_1$-approximation of a balanced asymmetric chain, the ordered limits $\{Z_k\}_{k=1,\ldots,s}$ in Theorem \ref{result-bal} exist. Set $p$ equal to the maximum index $k$ such that
  \begin{equation}
    Z_k \leq \max \{ \lim_{n \rightarrow \infty}Y_i(n) | i \in I_1 \}
  \end{equation}
  and follow steps (\ref{asd4}) to (\ref{asd7}) in Theorem \ref{ergodic}. Since, by definition of the island $I$:
  \begin{equation}
    \sum_{n=0}^{\infty}\sum_{i\in \bar{I}_1,j\in I_1} b_{ij}(n) = \infty,
  \end{equation}
  the RHS of inequality (\ref{asd7}) is unbounded as in the proof of Theorem \ref{ergodic}, which is a contradiction. Therefore all agents contained in every island end up in the same consensus cluster. Since the initial state was arbitrary, we obtain that every subchain is ergodic. From ergodicity and balanced asymmetry of each subchain, we conclude that the absolute infinite flow property holds for each subchain, i.e., inside each island.
\end{proof}


As a result of Theorem \ref{result-semi}, the following result, stated and proved previously in \cite{Touri:11c}, provides a sufficient condition for class-ergodicity of a chain of row stochastic matrices. Recall definitions of self-confidence and cut-balance properties from Part \ref{bal-asym}.

\begin{theorem}
  If chain $(A_n)$ is an $l_1$-approximation of a self-confident and cut-balanced chain, it is also class-ergodic.
  \label{self-type}
\end{theorem}

\begin{proof}
  From Proposition \ref{partial ergodicity}, to prove class-ergodicity of $(A_n)$, we can assume that $(A_n)$ is self-confident and cut-balanced. These two properties of $(A_n)$ imply that $(A_n)$ is balanced asymmetric. Therefore, according to Theorem \ref{result-semi}, it suffices to show that the absolute infinite flow property holds over each island of the strong interaction digraph $G_A$. Let $I$ be an arbitrary island and $T(0),T(1),\ldots$ be an arbitrary sequence of subsets of $I$ with the same cardinality. Keeping in mind Remark \ref{rmrk2}, we consider the following two cases:

  Case I. The sequence $T(0),T(1),\ldots$ becomes invariant after a finite time, i.e., there exist $T \subset I$ and $N \geq 0$ such that $T(i)=T$ for every $n \geq N$. In this case,
  \begin{equation}
    \sum_{n=0}^{\infty} \sum_{i \in \bar{T}(n+1)} \sum_{j \in T(n)} a_{ij}(n) \geq \sum_{n=N}^{\infty} \sum_{i \in \bar{T}} \sum_{j \in T} a_{ij}(n).
    \label{ya khoda}
  \end{equation}
  Since $I$ is a strongly connected component of the unbounded interactions graph, there exist two agents $p \in \bar{T}$ and $q \in T$ such that $\sum_{n=0}^{\infty} a_{pq}(n) = \infty$. Consequently, $\sum_{n=N}^{\infty} a_{pq}(n)$ diverges and so does the RHS of inequality (\ref{ya khoda}). This proves the result.

  Case II. The sequence $T(0),T(1),\ldots$ does not converge, i.e., there exists a time subsequence $n_0,n_1,\ldots$ such that $T(n_k) \neq T(n_k+1)$ for every $k=0,1,\ldots$. Clearly,
  \begin{equation}
      \sum_{n=0}^{\infty} \sum_{i \in \bar{T}(n+1)} \sum_{j \in T(n)} a_{ij}(n) \geq \sum_{k=0}^{\infty} \sum_{i \in \bar{T}(n_k+1)} \sum_{j \in T(n_k)} a_{ij}(n_k).
    \label{ya khodaa}
  \end{equation}
  Since $T(n_k) \neq T(n_k+1)$ and the two subsets are of the same cardinality, there exists an agent that belongs to both $\bar{T}(n_k+1)$ and $T(n_k)$. Hence, due to self-confidence of chain $(A_n)$, we have
  \begin{equation}
   \sum_{i \in \bar{T}(n_k+1)} \sum_{j \in T(n_k)} a_{ij}(n_k) > \delta >0.
  \end{equation}
  Therefore, the RHS of inequality (\ref{ya khodaa}) diverges. This proves the result again in this case.
\end{proof}


\section{Discussion}


\subsection{Relationship to Previous Work}
Considering the body of the work on discrete time linear consensus algorithms and their convergence properties in the past decade, \cite{Hend:11} and \cite{Touri:10a,Touri:10b,Touri:3,Touri:11c} appear to provide the most general results. In the following, we compare our results to those of the mentioned papers in terms of generality.

In Theorem 2 of \cite{Hend:11} (main discrete time result of \cite{Hend:11}), the authors require the following three assumptions to establish unconditional multiple consensus in system (\ref{modelss}): (i) A uniform positive lower bound on positive interaction rates, (ii) Positive diagonal coefficients, (iii) Cut-balance (discrete time version).

The above assumptions are stronger than the ones made in our Theorem \ref{self-type}, itself a consequence of our main result, Theorem \ref{result-semi}. More specifically, the self-confidence property in Theorem \ref{self-type} is implied by assumptions (i) and (ii), and the cut-balance property in Theorem \ref{self-type} is an immediate result of assumptions (i) and (iii).

In papers \cite{Touri:10a,Touri:10b,Touri:3,Touri:11c}, there are several related results, mostly extended to random chains. Among all the related results, one can consider Corollary 4 and the deterministic counterpart of Theorem 4 in \cite{Touri:11c}, as the most general ones. Corollary 4 in \cite{Touri:11c} is as general as our current Theorem 4. However, it is difficult a priori to rank in terms of generality the deterministic counterpart of Theorem 4 of \cite{Touri:11c} and our main results here, namely Theorems \ref{ergodic} and \ref{result-semi}. To see this, we note that there are example systems covered by our theorems and not by those in \cite{Touri:11c}, and vice versa.

As an illustration, if we define chain $(A_n)$ by:
\begin{equation}
  A_n =
  \begin{bmatrix}
	  1/n & 1-1/n \\ 1-1/n & 1/n
	\end{bmatrix}, \,\,\, \forall n \geq 0
\end{equation}
then, $(A_n)$ is balanced assymmetric and has the absolute infinite flow property. Thus, it is ergodic according to (our) Theorem \ref{ergodic}. However, ergodicity is not implied from Corollary 4 in \cite{Touri:11c}, as the latter would require $(A_n)$ to be weakly aperiodic.

Conversely, consider time-invariant chain $(A_n)$ defined by
\begin{equation}
  A_n =
  \begin{bmatrix}
	  1/2 & 1/2 \\ 1 & 0
	\end{bmatrix}, \,\,\, \forall n \geq 0.
\end{equation}
It is not balanced asymmetric and therefore it is outside the reach of our theorems. However, from Corollary 4 in \cite{Touri:11c}, one can conclude that $(A_n)$ is ergodic, since it belongs to the thus-denoted set $\mathcal{P}^*$ (see \cite{Touri:11c}) and is weakly aperiodic.


\subsection{Relationship to Known Models}

We now apply our theorems to chains corresponding to different types of models and consensus algorithms found in the literature in order to analyze when their transition chains become ergodic or class-ergodic.

\subsubsection{Models with Finite Range Interactions}

The Krause model \cite{Krause:97} is an example of endogenous models with finite range interactions. These models are special cases of first order models in which interaction rates evolve endogenously. In these models, agent $i$ receives information from agent $j$ if and only if the distance between the two agents is less than some pre-specified value $R$. More specifically, starting from some non increasing function $f:\mathbb{R}^{\geq 0} \rightarrow \mathbb{R}^{\geq 0}$ vanishing at $R$, define interaction weight
\begin{equation}
 a_{ij} = \frac{f(|X_i-X_j|)}{\sum_{k=1}^s f(|X_i-X_k|)}
 \label{hoof}
\end{equation}
Note that in the Krause model, $f(y) = 1$ for $0 \leq x < R$ and $f(y) = 0$ elsewhere. It can be proved that in this case, the transition chain has self-confidence with $\delta = 1/s$, and is cut-balanced with $M=s$ (see Part IV-B of \cite{Bolouki:12} for details). Thus, according to Theorem \ref{self-type}, the chain is class-ergodic, i.e., unconditional multiple consensus occurs.


\subsubsection{The C-S model}

The C-S (Cucker-Smale) model \cite{Cucker:07} is an example of endogenous consensus models with interaction weights remaining strictly positive. We apply our results to a generalized version of the C-S model \cite{Cucker:07} that describes evolution of positions $X_i$'s and velocities $V_i$'s in a bird flock, in a three dimensional Euclidian space:
\begin{equation}
 \begin{cases}
  X_i(n+1) & \hspace{-.1in} = X_i(n) + h V_i(n),\\
  V_i(n+1) & \hspace{-.1in} = V_i(n) + \sum_{j\neq i} f (\|X_i(n) - X_j(n) \|) (V_j(n) - V_i(n)),
 \end{cases}
 \label{cucker-smale model}
\end{equation}
where $f : \mathbb{R}^{\geq 0} \rightarrow \mathbb{R}^{\geq 0}$ is a non increasing function. Note that in this model, the limiting behavior of velocities is of interest. The transition chain in this algorithm can be obtained by rewriting velocities' update equation in the form of Eq. (\ref{modelss}). 
Clearly, the transition chain is symmetric and so is cut-balanced. To enforce self-confidence, one may require an additional assumption, such as, $f(y) < 1/s$, $\forall y \geq 0$. Under this assumption, the self-confidence property is satisfied with $\delta = 1/s$. The combination of the self-confidence and cut-balance properties of the chain allows an application of Theorem \ref{self-type} to conclude that the chain is class-ergodic (unconditional multiple consensus), without any additional assumption.
\begin{theorem}
 For a system with $s$ agents evolving according to generalized C-S dynamics (\ref{cucker-smale model}), assume that $f(y)$ has the following property:
 \begin{equation}
  f(y) < 1/s, \,\, \forall s \geq 0
  \label{f-bound}
 \end{equation}
 Also assume that initial agents positions and velocities are such that
 \begin{equation}
  M_v < \frac{s}{3 h} \int_{M_x}^{\infty} f(y) dy.
  \label{assumption-disc}
 \end{equation}
 where $M_x$ and $M_v$ are maximum norms of respectively initial agent position vector differences, and initial agent velocity vector differences. Then, all agents' velocities converge to a common value. Moreover, the maximum distance between any two agents remains bounded by some number $R$ at all times.
 \label{cs_erg}
\end{theorem}
Theorem \ref{cs_erg} is not an immediate result of Theorems \ref{ergodic} and \ref{self-type}. However, to prove Theorem \ref{cs_erg}, we employ a technique similar to that used in the proof of Theorem \ref{ergodic}. See Part IV-C of \cite{Bolouki:12} for the proof.

The following corollary follows from Theorem \ref{cs_erg}:
\begin{corollary}
 For a multi-agent system with dynamics (\ref{cucker-smale model}) where $f(y) = K/(\sigma^2 + y^2)^{\beta}$ (the C-S model \cite{Cucker:07}), assume that $K / \sigma^{2\beta} < 1/s$. Then, under either condition (i) or (ii) in the following, agents' velocities asymptotically converge to a common value: (i) $\beta \leq 1/2$, (ii) $\beta > 1/2$ and $M_v < sK / (3h(2\beta -1)(M_x + \sigma)^{2\beta -1})$.
 \end{corollary}


\subsubsection{The JLM model}

Similary, and without any additional assumptions, based on Theorem \ref{self-type}, it can be shown that in the JLM model \cite{Jadba:03}, multiple consensus occurs. Moreover, as the JLM model is balanced asymmetric, Theorem \ref{ergodic} gives a necessary and sufficient condition for unconditional consensus, although exponential convergence is not guaranteed (see \cite{Bolouki:12} for details).


\section{Conclusion}

In this note, we have focused on a class of linear distributed averaging algorithms in discrete time, such that the underlying non homogeneous update Markov chain satisfies a property called balanced asymmetry. Under the balanced asymmetry assumption, we established that, asymptotically, states of agents involved in the consensus algorithm keep taking their values within a fixed set of limiting values of cardinality at most $s$.

Furthermore, considering the graph of unbounded interactions and its islands as introduced by Hendrickx et al \cite{Hend:11} for continuous time consensus algorithms, under the balanced asymmetry assumption, we established a necessary and sufficient condition for the above limits to become that of individual agent states; the number of potential consensus clusters is equal to the number of islands, and consensus over an island occurs if and only if the so-called absolute infinite flow property (Touri and Nedi\'c \cite{Touri:3}) holds on that island. Finally, we displayed the applicability of our results to a number of well-known consensus models in the literature. In future work, we shall investigate the impact of the number of agents increasing to infinity on all of our results.


\end{document}